\documentclass[11pt]{amsart}
\usepackage{verbatim}

\newtheorem{theorem}{Theorem}[section]
\newtheorem{proposition}[theorem]{Proposition}
\newtheorem{prop}[theorem]{Proposition}

\newtheorem{claim}[theorem]{Claim}

\newtheorem{thm}{Theorem}[section]

\newtheorem{defn}[thm]{Definition}

\newcommand{\ZZ}{\mathbb{Z}}
\newtheorem{lem}[thm]{Lemma}
\newtheorem{cor}[theorem]{Corollary}
\newtheorem{question}[theorem]{Question} 
\newtheorem{definition}[theorem]{Definition}

\newtheorem{conjecture}[theorem]{Conjecture}

\theoremstyle{plain}
\numberwithin{equation}{theorem}

\theoremstyle{remark}
\newtheorem*{notation}{Notation}
\newtheorem{remark}[theorem]{Remark}
\newtheorem{example}[theorem]{Example}

\newcommand{\C}{{\mathbb C}}

\DeclareMathOperator{\Id}{Id}

\newcommand{\N}{{\mathbb N}}

\newcommand{\Z}{{\mathbb Z}}
\newcommand{\bA}{{\mathbb A}}

\newcommand{\cX}{{\mathcal X}}
\newcommand{\bcX}{{\overline \cX}}
\newcommand{\D}{{\mathbb D}}

\newcommand{\Dbar}{\overline{\D}}

\newcommand{\cOO}{\overline{\cO}}

\def\Frac{\operatorname{Frac}}

\newcommand{\bcP}{\overline \Phi}
\DeclareMathOperator{\Spec}{Spec}

\DeclareMathOperator{\fa}{\mathfrak{a}}
\DeclareMathOperator{\fb}{\mathfrak{b}}
\DeclareMathOperator{\fc}{\mathfrak{c}}

\newcommand{\bP}{{\mathbb P}}
\newcommand{\bZ}{{\mathbb Z}}

\newcommand{\bC}{{\mathbb C}}

\newcommand{\cA}{{\mathcal A}}
\newcommand{\bQ}{{\mathbb Q}}
\newcommand{\bF}{{\mathbb F}}

\newcommand{\lra}{\longrightarrow}

\newcommand{\fm}{\mathfrak m}
\newcommand{\fmh}{\widehat \fm}
\newcommand{\fmb}{\overline \fm}
\newcommand{\fp}{\mathfrak p}
\newcommand{\fq}{\mathfrak q}
\newcommand{\cO}{\mathcal{O}}
\newcommand{\cF}{\mathcal{F}}

\newcommand{\cV}{\mathcal{V}}

\newcommand{\smooth}{{\rm smooth}}

\newcommand{\Oxh}{\widehat \cO_{\cX,x}}
\newcommand{\Ox} {\cO_{\cX,x}}

\title[The dynamical Mordell-Lang problem]{The dynamical Mordell-Lang
  problem for \'{e}tale maps}
\author{J.~P.~Bell, D.~Ghioca, and T.~J.~Tucker}
\keywords{Mordell-Lang conjecture, dynamics}
\subjclass[2000]{Primary 14K12,Plea Secondary 37F10}

\address{Jason Bell\\
Department of Mathematics\\
Simon Fraser University\\
Burnaby, BC V5A 1S6\\
}

\email{jpb@math.sfu.ca}

\address{
Dragos Ghioca \\
Department of Mathematics \& Computer Science\\
University of Lethbridge \\
Lethbridge, AB T1K 3M4 
}

\email{dragos.ghioca@uleth.ca}

\address{
Thomas Tucker\\
Department of Mathematics\\
Hylan Building\\
University of Rochester\\
Rochester, NY 14627
}

\email{ttucker@math.rochester.edu}

\begin{document}

\begin{abstract}
  We prove a dynamical version of the Mordell-Lang conjecture for
  \'{e}tale endomorphisms of quasiprojective varieties. We use $p$-adic methods inspired by the work of Skolem, Mahler, and Lech, combined with methods from algebraic geometry. As special cases of our result we obtain a new proof of the classical Mordell-Lang conjecture for cyclic subgroups of a semiabelian variety, and we also answer positively a question of Keeler/Rogalski/Stafford for critically dense sequences of closed points of a Noetherian integral scheme.
\end{abstract}
\thanks {The second author was partially supported by NSERC.
The third author was partially supported by NSA
    Grant 06G-067 and NSF Grant DMS-0801072.}

\maketitle

\section{Introduction}
\label{intro}

Let $X$ be a quasiprojective variety over the complex numbers $\bC$, let $\Phi: X \lra
X$ be a morphism, and let $V$ be a closed subvariety of
$X$.  For any integer $i\geq 0$, denote by $\Phi^i$ the
$i^{\text{th}}$ iterate $\Phi\circ\cdots\circ\Phi$; for any point $\alpha\in X(\bC)$, we let $\cO_{\Phi}(\alpha):=\{\Phi^i(\alpha)\text{ : }i\in\N\}$ be the (forward) $\Phi$-orbit of $\alpha$.
If $\alpha \in X(\bC)$ has the property that there is some
integer $\ell\geq 0$ such that $\Phi^\ell(\alpha) \in W(\bC)$,
where $W$ is a periodic subvariety of $V$, then there are infinitely
many integers
$n\geq 0$ such that $\Phi^n(\alpha) \in V$.  More precisely, if
$N\geq 1$ is the period of $W$ (the smallest positive integer $j$ for which
$\Phi^j(W) = W$), then $\Phi^{kN + \ell}(\alpha) \in W(\bC)
\subseteq V(\bC)$ for all integers $k\geq 0$. In \cite{Denis-dynamical}, Denis asked the following question.

\begin{question}\label{general}
  If there are infinitely many nonnegative integers $m$ such that
  $\Phi^m(\alpha) \in V(\bC)$, are there necessarily integers $N\geq
  1$ and $\ell\geq 0$ such that $\Phi^{kN + \ell}(\alpha) \in V(\bC)$
  for all integers $k\geq 0$?
\end{question}
Note that if $V(\bC)$ contains an infinite set of the form $\{\Phi^{kN+\ell}(\alpha)\}_{k\in\N}$ for some positive integers $N$ and $\ell$, then $V$ contains a positive dimensional subvariety invariant under $\Phi^N$ (simply take the union of the positive dimensional components of the Zariski closure of $\{\Phi^{kN+\ell}(\alpha)\}_{k\in\N}$).

Denis \cite{Denis-dynamical} showed that the answer to
Question~\ref{general} is ``yes'' under the additional hypothesis that
the integers $n$ for which $\Phi^n(\alpha) \in V(\bC)$ are
sufficiently dense in the set of all positive integers; he also
obtained results for automorphisms of projective space without using
this additional hypothesis.  Later the problem was solved completely
in \cite{Bell} in the case of automorphisms of affine varieties $X$,
by showing that the set of all $n\in\N$ such that $\Phi^n(\alpha)\in
V(\bC)$ is a union of at most finitely many arithmetic progressions,
and of at most finitely many numbers.  In \cite{newlog}, the following
conjecture was proposed.
\begin{conjecture}
\label{dynamical M-L}
Let $X$ be a quasiprojective variety defined over $\C$, let $\Phi:X\lra X$ be an endomorphism, let $V$ be a subvariety of $X$, and let $\alpha\in X(\C)$. Then the intersection $V(\C)\cap \cO_{\Phi}(\alpha)$ is a union of at most finitely many orbits of the form $\cO_{\Phi^N}(\Phi^{\ell}(\alpha))$, for some nonnegative integers $N$ and $\ell$.
\end{conjecture}
Note that the orbits for which $N=0$ are singletons, so that the
conjecture allows not only infinite forward orbits but also finitely
many extra points. We view our Conjecture~\ref{dynamical M-L} as a
dynamical version of the classical Mordell-Lang conjecture, where
subgroups of rank one are replaced by orbits under a morphism. 

Results were obtained in the case when
$\Phi: \bA^2 \lra \bA^2$ takes the form $(f,g)$ for $f,g \in \bC[t]$ and the subvariety $V$ is a line (\cite{Mike}), and in the case when
$\Phi:\bA^g \lra \bA^g$ has the form $(f,\dots,f)$ where $f \in K[t]$ (for a number field $K$)
has no periodic critical points other than the point at infinity
(\cite{Par}). Also, in \cite{newlog} a general approach to Conjecture~\ref{dynamical M-L} was developed in the case the orbit $\cO_{\Phi}(P)$ intersects a sufficiently small $p$-adic neighborhood of a $\Phi$-periodic point of $X$ where the Jacobian of $\Phi$ is diagonalizable. Furthermore, in \cite{newlog}, Conjecture~\ref{dynamical M-L} was proved if $X$ is a semiabelian variety and $\Phi$ is an algebraic group endomorphism.

The technique used in \cite{Bell}, \cite{Par} and \cite{newlog} is a modification of
a method first used by Skolem \cite{Skolem} (and later extended by
Mahler \cite{Mahler-2} and Lech \cite{Lech}) to treat linear
recurrence sequences.  The idea is
to show that there is a positive integer $N$ such that for each $i = 0,\dots, N-1$ there is a $p$-adic analytic map
$\theta_i$ on $\bZ_p$ such that $\theta_i(k) =
\Phi^{kN + i}(\alpha)$ for all sufficiently large integers $k$.  Given any polynomial
$F$ in the vanishing ideal of $V$, one thus obtains a $p$-adic
analytic function $F \circ \theta_i$ that vanishes on all $k$ for which
$\Phi^{kN + i}(\alpha) \in V$.  Since an analytic function cannot have
infinitely many zeros in a compact subset of its domain of convergence unless that
function is identically zero, this implies that if there are
infinitely many $n \equiv i \pmod{N}$ such that $\Phi^n(\alpha) \in V$,
then $\Phi^{kN + i}(\alpha) \in V$ for all $k$ sufficiently large.

In the case of \cite{Bell}, the existence of the $p$-adic analytic maps $\theta_i$ is proved using properties of automorphisms of the affine plane, such as having constant determinant for their Jacobian. In \cite{Par}, the existence of the $p$-adic analytic maps
$\theta_i$ is proved by using linearizing maps developed by
Rivera-Letelier \cite{Riv1}, while in \cite{newlog}, the existence of the $\theta_i$'s is proved using work of Hermann and Yoccoz (see \cite{HY}).
In this paper, using methods from arithmetic geometry and $p$-adic analysis we surpass all of the above results, and we prove Conjecture~\ref{dynamical M-L} in the case $\Phi$ is any \'{e}tale map. 

\begin{theorem}
\label{real main result}
Let $\Phi:X \lra X$ be an \'{e}tale endomorphism of any quasiprojective
variety defined over $\bC$. Then for any subvariety $V$ of $X$, and for any point $\alpha\in X(\bC)$ the intersection $V(\C)\cap \cO_{\Phi}(\alpha)$ is a union of at most finitely many orbits of the form $\cO_{\Phi^N}(\Phi^{\ell}(\alpha))$ for some $N,\ell\in\N$.
\end{theorem}

Our result provides a positive answer to a question raised in
\cite{KeeRogSta} regarding critically dense orbits under automorphisms
of integral Noetherian schemes (see
Section~\ref{applications}). Theorem~\ref{real main result} has the
following interesting corollary.
\begin{cor}
  \label{alternative formulation}
  Let $X$ be an irreducible quasiprojective variety, let $\Phi:X\lra
  X$ be an \'{e}tale endomorphism, and let $\alpha\in X(\C)$. If the
  orbit $\cO_{\Phi}(\alpha)$ is Zariski dense in $X$, then any proper
  subvariety of $X$ intersects $\cO_{\Phi}(\alpha)$ in at most
  finitely many points.
\end{cor}

Our result fits into Zhang's far-reaching system of dynamical
conjectures (see \cite{ZhangLec}). Zhang's conjectures include
dynamical analogues of Manin-Mumford and Bogomolov conjectures for
abelian varieties (now theorems of \cite{Raynaud1,Raynaud2},
\cite{Ullmo} and \cite{Zhang}). One of the conjectures from
\cite{ZhangLec} asks that any irreducible projective variety $X$
defined over a number field $K$ has a point in $X({\overline K})$ with
a Zariski dense orbit under any ``polarizable'' endomorphism $\Phi$
(Zhang defines a polarizable endomorphism $\Phi$ as one for which there
exists an ample line bundle $L$ such that $\Phi^* L \cong L^{\otimes
  r}$ for some $r>1$).  Note in particular that any Zariski dense
orbit must avoid all proper $\Phi$-periodic subvarieties of $X$.
Theorem~\ref{real main result} says that \emph{any} subvariety of $X$
containing infinitely many points of a $\Phi$-orbit \emph{must}
contain a $\Phi$-periodic subvariety. We hope that Theorem~\ref{real
  main result} represents real progress towards proving
Conjecture~\ref{dynamical M-L}.

We now briefly sketch the plan of our paper. In Section~\ref{geometry} we present the geometric setup, while in Section~\ref{analytic} we derive the existence of certain analytic functions which are used later to prove Theorem~\ref{main result}. In particular, our results from Section~\ref{analytic} provide generalizations of Rivera-Letelier's results \cite{Riv1} regarding analytic conjugation maps corresponding to quasiperiodic domains of rational $p$-adic functions. In Section~\ref{main proof} we prove our main result, while in Section~\ref{applications} we present several interesting applications of our Theorem~\ref{real main result} for automorphisms of Noetherian integral schemes. In particular, we obtain a new proof of the classical Mordell-Lang conjecture for cyclic subgroups, and we provide a positive answer to \cite[Question 11.6]{KeeRogSta}.

\begin{notation} 
  We write $\N$ for the set of nonnegative integers. If $K$ is a
  field, we write $\overline{K}$ for an algebraic closure of $K$.
  Given a prime number $p$, we denote by $\mid\cdot\mid_p$ the usual
  absolute value on $\bQ_p$; that is, we have $|p|_p = 1/p$.  When we
  work in $\bQ_p^g$ with a fixed coordinate system, then, for
  $\vec{\alpha} = (\alpha_1\dots, \alpha_g) \in\bQ_p^g$ and $r>0$, we
  write $\D(\vec{\alpha},r)$ for the open disk of radius $r$ in $\bQ_p^g$
  centered at $\alpha$.  More precisely, we have
  $$
  \D(\vec{\alpha},r) := \{ (\beta_1, \dots, \beta_g) \in \bQ_p^g \;
  \mid \; \max_i |\alpha_i - \beta_i|_p < r \}.$$
  Similarly, we let
  $\Dbar(\vec{\alpha},r)$ be the {\it closed} disk of radius $r$
  centered at $\vec{\alpha}$.  In the case where $g=1$, we drop the
  vector notation and denote our discs as $\D(\alpha,r)$ and
  $\Dbar(\alpha,r)$.  We say that a function $F$ is {\it (rigid) analytic} on
  $\D(\alpha,r)$ (resp. $\Dbar(\alpha, r)$) if there is a power series
  $\sum_{n=0}^\infty a_n (z-\alpha)^n$, with coefficients in $\bQ_p$,
  convergent on all of $\D(\alpha,r)$ (resp.  $\Dbar(\alpha, r)$) such
  that $F(z) = \sum_{n=0}^\infty a_n (z-\alpha)^n$ for all $z \in
  \D(\alpha,r)$ (resp. $\Dbar(\alpha, r)$). Similarly, we define convergence of an analytic function $f$ on $\D(\vec{\alpha},r)$ (resp. $\Dbar(\vec{\alpha},r)$) where $\vec{\alpha}\in\bQ_p^g$ for any positive integer $g$.

Finally, all subvarieties in our paper are closed subvarieties.
\end{notation}

\section{Preliminary results from arithmetic geometry}
\label{geometry}

In this Section we construct the geometric setup for the proof of Theorem~\ref{main result} which is the main building block for proving Theorem~\ref{real main result}. More precisely, Theorem~\ref{main result} deals with the case when $\Phi$ is an unramified endomorphism of an irreducible smooth  quasiprojective variety $X$. In the proof of Theorem~\ref{main result} we will show that we may assume $X$ has a model $\cX$ over $\Z_p$ (for a suitable prime $p$) such that $\Phi:\cX\lra \cX$ is an unramified map. The goal of this Section is to construct a suitable analytic function associated to $\Phi$ which maps the residue class of a closed point $x\in\cX$ into itself.

\subsection{Notation}
\begin{itemize}
\item  $\cX$ is a
quasiprojective scheme over $\bZ_p$ such that both the generic fiber and
special fiber are geometrically irreducible varieties (over $\bQ_p$ and $\bF_p$, respectively);
\item $\bcX$ is the closed fiber of $\cX$ (i.e., it is $\cX
  \times_{\bZ_p} {\bF_p}$); 
\item $X$ is the generic fiber of $\cX$ (i.e., it is $\cX
  \times_{\bZ_p} {\bQ_p}$);
\item $\Phi: \cX \lra \cX$ is an unramified map of $\bZ_p$-schemes; 
\item $\bcP: \bcX \lra \bcX$ is the restriction of $\Phi$ on the closed fiber;
\item $r:\cX(\bZ_p) \lra \bcX(\bF_p)$ is the usual reduction map;
\item $x$ is an $\bF_p$-point on $\cX$ in the smooth locus of the
  projection to $\bZ_p$ such that there is a point $\alpha \in
  \cX(\bZ_p)$ for which $r(\alpha) = x$.

\end{itemize}

\subsection{Completions of local rings at smooth closed points}
  Let $\cO_{\cX,x}$ be the local ring of $x$ as a point on $\cX$ and let ${\widehat \cO_{\cX,x}}$ be the completion of
$\cO_{\cX,x}$ at its maximal ideal $\fm$; let $\fmh$ be the maximal ideal in $\cO_{\cX,x}$.  The Cohen structure
theorem (see \cite[Section 29]{Mats2} or \cite[Chapter IX]{Bour2}) then
gives the following.

\begin{prop}\label{regular}
There are elements $T_1,\dots,T_g$ of $\Oxh$ such that
$$ \Oxh = \bZ_p [[T_1, \dots, T_g]].$$
\end{prop}
\begin{proof}
  Since $x$ is smooth, $\cO_{\cX,x}$ must be a regular ring.
  Therefore, its completion $\Oxh$ must be regular as well by
  \cite[Proposition 11.24]{AM}.  The existence of a point in
  $\cX(\bZ_p)$ that reduces to $x$ means that there is a surjective
  map $g:\cO_{\cX,x} \lra \bZ_p$, which extends to a map ${\hat g}:
  \Oxh \lra \bZ_p$. Because ${\hat g}$ is surjective, we conclude that
  $p\notin\fmh^2$; thus $\Oxh$ is unramified, in the terminology of
  \cite[Section 29]{Mats2}.  Theorem $29.7$ of \cite{Mats2} states
  that any unramified complete regular Noetherian ring of
  characteristic $0$ with a finite residue field is a formal power
  series ring over a complete $p$-ring. By \cite[Corollary,
  p. 225]{Mats2}, $\bZ_p$ is the only complete $p$-ring with residue
  field $\bF_p$ (note that $\cO_{\cX,x}$ has residue field $\bF_p$
  because we have a surjective map, induced by $g$, from the residue
  field of $\cO_{\cX,x}$ onto $\bF_p$).
\end{proof}

There is a one-to-one correspondence between the points in
$\cX(\bZ_p)$ that reduce to $x$ and the primes $\fp$ in $\Ox$ such
that $\Ox/\fp \cong \bZ_p$.  For each such prime $\fp$, its completion
${\hat \fp}$ in $\Oxh$ has the property that $\fp \Oxh = {\hat \fp}$ (see
\cite[Theorem 8.7]{Mats2}).  Furthermore, ${\hat \fp}$ is a prime ideal in $\Oxh$
with residue domain $\bZ_p$
since the sequence
$$ 0 \lra {\hat \fp} \lra \Oxh \lra \bZ_p \lra 0$$
is exact; this follows from the fact that $\Oxh$ is flat over $\Ox$
(\cite[Theorem 8.8]{Mats2}) along with the fact that the quotient $\Ox
/ \fp \cong \bZ_p$ is complete with respect to the $\fm$-adic
topology.  Thus, if $\fq$ is any prime in $\Oxh$ with residue domain
$\bZ_p$ then $\fq$ must be the completion of $\fq \cap \Ox$, because
$\dim \Ox = \dim \Oxh$ (\cite[Corollary 11.19]{AM}).  Hence, we have a
one-to-one correspondence between the points in $\cX(\bZ_p)$ that
reduce to $x$ and the primes $\fq$ in $\Oxh$ such that $\Oxh/\fq \cong
\bZ_p$.  Note that primes $\fq$ in $\Oxh$ for which $\Oxh/\fq \cong
\bZ_p$ are simply the ideals of the form $(T_1 - p z_1, \dots, T_g - p
z_g)$ where the $z_i$ are in $\bZ_p$.  For each $\bZ_p$-point $\beta$
in $\cX$ such that $r(\beta) = x$, we write $\iota(\beta) =
(\beta_1,\dots, \beta_g)$ where $\beta$ corresponds to the prime ideal
$$ (T_1 - p \beta_1, \dots, T_g - p \beta_g)$$
in $\Oxh$. Note that $\iota^{-1}:\bZ_p^g\lra \cX(\bZ_p)$ induces an analytic bijection between $\bZ_p^g$ and the analytic neighborhood of $\cX(\bZ_p)$ consisting of points $\beta$ such that $r(\beta)=x$.

\begin{prop}\label{power}
  Suppose that $\bcP(x) = x$.  Then there are power series $F_1,
  \dots, F_g \in \bZ_p[[U_1, \dots, U_g]]$ such that
\begin{enumerate}
\item each $F_i$ converges on $\bZ_p^g$;
\item for each $\beta\in\cX(\bZ_p)$ such that $r(\beta)=x$, we have
\begin{equation}\label{power eq}
i(\Phi(\beta)) =   (F_1(\beta_1,\dots, \beta_g),\dots, F_g(\beta_1,
\dots, \beta_g)); \text { and }
\end{equation}
\item each $F_i$ is congruent to a linear polynomial mod $p$ (in
  other words, all the coefficients of terms of degree greater than one
  are divisible by $p$).   
\end{enumerate}
\end{prop}
\begin{proof}
The map $\Phi$ induces a ring homomorphism 
$$
\Phi^*: \Oxh \lra \Oxh$$
that sends the maximal ideal $\fmh$ in $\Oxh$ to
itself.  For each $i$, there is a power series $H_i \in \bZ_p[[T_1,
\dots, T_g]]$ such that $\Phi^*T_i = H_i$. Furthermore, since
$\Phi^*T_i$ must be in the maximal ideal of $\Oxh$, the constant term
in $H_i$ must be in $p\bZ_p$.  Then, for any
$(\alpha_1,\dots,\alpha_g) \in p \bZ_p$, we have
$$(\Phi^*)^{-1}(T_1 - \alpha_1, \dots, T_g - \alpha_g) = (T_1 - H_1(\alpha_1, \dots, \alpha_g),\dots, T_g - H_g(\alpha_1,
\dots, \alpha_g))$$
since 
$$(T_1 - H_1(\alpha_1, \dots, \alpha_g),\dots, T_g - H_g(\alpha_1,
\dots, \alpha_g))$$
is a prime ideal of coheight equal to one, and 
$$
H_i(T_1, \dots, T_g) - H_i(\alpha_1, \dots, \alpha_g)$$
is in the
ideal $(T_1 - \alpha_1, \dots, T_g - \alpha_g)$ for each $i$.  Thus,
if $\beta$ corresponds to the prime ideal
$$ (T_1 - p \beta_1, \dots, T_g - p \beta_g)$$
then $\Phi(\beta)$ corresponds to the prime ideal
$$
(T_1 - H_1(p \beta_1, \dots, p \beta_g), \dots, T_g - H_g(p
\beta_1, \dots, p \beta_g)).$$
Hence, letting 
$$F_i(T_1, \dots, T_g) := \frac{1}{p} H_i(p T_1, \dots, p T_g)$$
gives the desired map. Since $H_i \in \bZ_p[[T_1, \dots, T_g]]$, it
follows that $F_i$ must converge on $\bZ_p$ and that all the
coefficients of terms of degree greater than one of $F_i$ are
divisible by $p$.  Since the constant term
in $H_i$ is divisible by $p$, we conclude that 
$$F_1, \dots, F_g \in  \bZ_p[[T_1, \dots, T_g]],$$
as desired. 
\end{proof}

Switching to vector notation, we write
$$ {\vec \beta} := (\beta_1, \dots, \beta_g)\in\bZ_p^g, $$
and we let
$$
\cF ({\vec \beta}) := (F_1(\beta_1,\dots,\beta_g), \dots,
F_g(\beta_1, \dots, \beta_g)).$$
From Proposition~\ref{power}, we see
that there is a $g \times g$ matrix $L$ with coefficients in
$\bZ_p$ and a constant ${\vec C} \in \bZ_p^g$ such that
\begin{equation}\label{un}
\cF({\vec \beta}) = \vec{C} + L({\vec \beta})  + \text{ higher
  order terms }
\end{equation}
Note that since all of the higher order terms are divisible by $p$,  we
also have
\begin{equation}\label{mod p}
\cF({\vec \beta}) \equiv \vec{C} +  L({\vec \beta}) \pmod{p}. 
\end{equation}

\begin{remark}
\label{divisible by many}
Moreover, using that $F_i(T_1,\dots,T_g)=\frac{1}{p} H_i(pT_1,\dots,pT_g)$, we obtain that for each $k_1,\dots,k_g\in\N$ such that $k_1+\dots+k_g\ge 2$, the coefficient of $T_1^{k_1}\cdots T_g^{k_g}$ in $F_i$ belongs to $p^{k_1+\dots +k_g-1}\cdot \bZ_p$.
\end{remark}

\begin{prop}
\label{prop:pI-0}
Suppose that $\bcP(x) = x$ and that $\bcP$ is
unramified at $x$.   Let $L$ be as in \eqref{un}.  Then $L$ is invertible modulo $p$.
\end{prop}
\begin{proof}
  Let $\cO_{\bcX, x}$ denote the local ring of $x$ on $\bcX$ and let
  $\fmb$ denote its maximal ideal.  Since $\bcP$ is unramified, the
  map $\bcP^*: \cO_{\bcX, x} \lra \cO_{\bcX, x}$ sends $\fmb$
  surjectively onto itself (see \cite[Appendix B.2]{BG}).  Thus in
  particular it induces an isomorphism on the $\bF_p$-vector space
  $\fmb/\fmb^2$.  Completing $\cO_{\bcX, x}$ at $\fmb$, we then get an
  induced isomorphism $\sigma$ on ${\widehat \fmb}/{\widehat \fmb}^2$,
  where ${\widehat \fmb}$ is the maximal ideal in the completion of
  $\cO_{\overline{\cX},x}$ at $\fmb$.  This isomorphism is obtained by
  taking the map $\Phi^*: \fmh / \fmh^2 \lra \fmh / \fmh^2$ and
  modding out by $p$, where $\fmh$ is the maximal ideal of $\Oxh$.
  Writing $\sigma$ as a linear transformation with respect to the
  basis $T_1,\dots, T_g$ for ${\widehat \fmb}/{\widehat \fmb}^2$, we obtain the dual of
  the reduction of $L$ mod $p$.  Thus, if $\bcP^*$ induces an isomorphism
  on $\fmb/\fmb^2$, then the reduction mod $p$ of $L$ itself must be
  invertible.
\end{proof}
\begin{remark}
Note that the dual of $\fmb/\fmb^2$ is the Zariski tangent space of
the point $x$ considered as a point on the special fiber ${\overline \cX}$ (see
\cite[page 80]{H}).  Thus the reduction of $L$ modulo $p$ is simply the
Jacobian of $\bcP$ at $x$.      
\end{remark}  

\begin{proposition}
\label{prop:pI}
There exists a positive integer $n$ such that $\cF^n(\vec{\beta}) \equiv \vec{\beta} \pmod{p}$ for each $\vec{\beta}\in\bZ_p^g$.
\end{proposition}
\begin{proof}
Because $L$ is invertible modulo $p$, it means that the reduction modulo $p$ of the affine map
$$\vec{\beta}\mapsto \vec{C} + L(\vec{\beta})$$
induces an automorphism of $\mathbb{F}_p^g$. Therefore, there exists a positive integer $n$ such that 
\begin{equation}
\label{pI}
\cF^n(\vec{\beta})\equiv\vec{\beta}\pmod{p},
\end{equation}
for all $\vec{\beta}\in\bZ_p$.
\end{proof}

\section{Construction of an analytic function}
\label{analytic}

The goal of this Section is to construct a $p$-adic analytic function $U:\Z_p\lra\Z_p^n$ such that $U(z+1)=\cF(U(z))$, where $\cF$ is constructed as in Section~\ref{geometry} for a closed point $x\in\cX$ and an unramified endomorphism $\Phi$ of the $n$-dimensional smooth $\Z_p$-scheme $\cX$. For this, we generalize the construction from \cite{Bell}, and thus provide the key analytical result (see our Theorem~\ref{thm: padic}) which will be used in the proof of Theorem~\ref{main result}.

\begin{defn} 
\label{BC}
Given a prime $p$, we let $B$ denote the ring of Mahler polynomials
and let $C$ denote the ring of Mahler series; i.e.,
$$B \ = \ \left\{ \sum_{i=0}^{m} c_i \binom{z}{i}~:~c_i\in \mathbb{Z}_p, m\ge 0\right\},$$~ $$C \ = \ \left\{ \sum_{i=0}^{\infty} c_i \binom{z}{i}~:~c_i\in \mathbb{Z}_p, |c_i|_p\rightarrow 0\right\}.$$
\end{defn}

\begin{remark}\label{rmk}
The names Mahler polynomials and Mahler series are used because of a
result of Mahler \cite{MahlerK1, MahlerK2}, which states that $C$ is
precisely the collection of continuous maps from $\mathbb{Z}_p$ to
itself.

More precisely, Mahler shows that $B$ is the ring of all polynomials
$f\in\bQ_p[z]$ such that $f(\bZ_p)\subset \bZ_p$, while $C$ is the
ring of all power series $g\in\bQ_p[[z]]$ which are convergent on
$\bZ_p$ and satisfy $g(\bZ_p)\subset\bZ_p$.
\end{remark}

\begin{lem} Let $n\in\N$ and let $p$ be a prime number.  Suppose that
\[ S_n \ = \ \Bigg\{ c+\sum_{i=1}^n p^i h_i(z)~|~c\in\bZ_p, h_i(z)\in B, {\rm deg}(h_i)\le 2i-1\Bigg\}\]
and
\[ T_n \ = \ S_n+\Bigg\{ \sum_{i=1}^{\infty} p^i h_i(z)~|~ h_i(z)\in B, {\rm deg}(h_i)\le 2i-2\Bigg\},\]
with the convention that if $n=0$, then $S_n=\bZ_p$.

Then the subalgebra of $C$ generated by convergent $\mathbb{Z}_p$-power series with variables in $S_n$ is contained in $T_n$.
\label{lem: SN}
\end{lem}
\begin{proof} 
For $n=0$, the conclusion is immediate. So, assume $n\ge 1$. 

Since $S_n$ and $T_n$ are both closed under addition, and $T_n$ is closed under taking limits of polynomials which are contained in $T_n$, and $S_n\subseteq T_n$, it is sufficient to show that $S_nT_n\subseteq T_n$.
To do this, suppose
\[H(z) \ = \ c_1+ \sum_{i=1}^n p^i h_i(z) \ \in \ S_n\] and
\[G(z) \ = \ c_2 + \sum_{i=1}^{\infty} p^i g_i(z) \ \in \ T_n,\]
where $c_1,c_2\in \bZ_p$ and $h_i(z),g_i(z)\in B$ with ${\rm deg}(g_i)\le 2i-2$ for $i>n$ and ${\rm deg}(g_i),{\rm deg}(h_i)\le 2i-1$ for $i\le n$.
We must show that $H(z)G(z)\in T_n$.  Notice that
\[
H(z)G(z) \ = \ -c_1c_2 + c_2H(z) +c_1G(z)  + \left(\sum_{i=1}^n p^i h_i(z)\right)\cdot\left(\sum_{j=1}^{\infty} p^j g_j(z)\right). \]
Then $(-c_1c_2 + c_2H(z)+c_1G(z))\in T_n$ and since $T_n$ is closed under addition, it is sufficient to show that
\[\left(\sum_{i=1}^n p^i h_i(z)\right)\cdot\left(\sum_{j=1}^{\infty} p^j g_j(z)\right)  \ = \ \sum_{k=2}^{\infty} p^k \sum_{i=1}^{k-1}g_i(z)h_{k-i}(z)\]
is in $T_n$.  But since $g_i(z)$ has degree at most $2i-1$ and $h_{k-i}(z)$ has degree at most $2(k-i)-1$, we see that 
$$\sum_{i=1}^{k-1} g_i(z)h_{k-i}(z)$$ has degree at most $2k-2$.  It follows that
\[\left(\sum_{i=1}^n p^i h_i(z)\right)\cdot\left(\sum_{j=1}^{\infty} p^j g_j(z)\right)  \ \in \ T_n.\]
This concludes the proof of Lemma~\ref{lem: SN}. 
\end{proof}

\begin{thm} \label{thm: padic} Let $n$ be a positive integer, let $p$ be a prime number and let $\varphi_1,\dots,\varphi_n\in \bZ_p[[x_1,\dots,x_n]]$ be convergent power series on $\bZ_p^n$ such that for each $i=1,\dots,n$ we have
\begin{enumerate}
\item[(a)] $\varphi_i(x_i) \equiv x_i \pmod{p}$; and 
\item[(b)] for each $k_1,\dots,k_n\in\N$ such that $k_1+\dots+k_n\ge 2$, the coefficient of $x_1^{k_1}\cdots x_n^{k_n}$ in the power series $\varphi_i$ belongs to $p^{k_1+\dots+k_n-1}\cdot \bZ_p$.
\end{enumerate}
Let $(\omega_1,\dots,\omega_n)\in\bZ_p^n$ be fixed.
If $p>3$, then there exist $p$-adic analytic functions
$f_1,\ldots,f_n\in \mathbb{Q}_p[[z]]$ such that for each $i=1,\dots,n$ we have
\begin{enumerate}
\item $f_i$ is convergent for $|z|_p \leq 1$;
\item $f_i(0)=\omega_i$;
\item $|f_i(z)|_p \le 1$ for $|z|_p \leq 1$; and
\item $f_i(z+1) = \varphi_i(f_1(z),\ldots,f_n(z))$.
\end{enumerate}
\end{thm}

A particular case of our result, when $n=1$ and $\varphi_1$ is a rational $p$-adic function, is discussed in \cite{Riv1}.

\begin{proof}[Proof of Theorem~\ref{thm: padic}.] 
We construct $(f_1(z), ..., f_{n}(z))$ by ``approximation''.  We let $B$ and 
$C$ be the ring of Mahler polynomials and the ring of Mahler series that are convergent for $|z|_p\le 1$ (as in Definition~\ref{BC}).
We prove that for each $j \ge 0$ there
exist polynomials $h_{i,j}(z) \in B$  ($1\le i\le n$) that satisfy
the three conditions:

(1) for each $i=1,\dots,n$, we have $h_{i,0}(0)=\omega_i$, while $h_{i,j}(0)=0$ for $j\ge 1$;

(2) $h_{i,j}(z)$ has degree at most $2j-1$ for $j\ge 1$ and $1\le i\le n$; and

(3) if $g_{i,j}(z)=\sum_{k=0}^{j} p^{k}h_{i,k}(z)$, then for each $i=1,\dots,n$ we have
\[g_{i,j}(z+1)\equiv \varphi_i(g_{1,j}(z),\ldots ,g_{n,j}(z)) ~(\bmod ~
p^{j+1}C ).\]

We define $h_{i,0}(z)=g_{i,0}(z)=\omega_i$ for $1\le i\le n$. Because $\varphi_i(x_1,\dots,x_n)\equiv x_i\pmod{p}$, and because each $\varphi_i$ is convergent on $\bZ_p^n$, we conclude that  for every $i=1,\dots,n$ we have
$$g_{i,0}(z+1)\equiv\varphi_i(g_{1,0}(z),\dots,g_{n,0}(z))\equiv \omega_i\pmod{pC}.$$

Let $j\ge 1$ and assume that we have defined $h_{i,k}$ for $1\le i\le n$ 
and $k <j $
so that conditions (1)-(3) hold. Our goal is now to construct polynomials
$h_{i, j}(z)\in B$,
so that conditions (1)-(3) hold.
By assumption
\begin{equation} \label{eq: gH}
g_{i,j-1}(z+1)-\varphi_i(g_{1,j-1}(z),\ldots ,g_{n,j-1}(z))=p^jQ_{i,j}(z),
\end{equation}
with
$Q_{i,j}\in C$
for $1\le i\le n$.
Using the notation of the statement of Lemma \ref{lem: SN}, we see that conditions (2) and (3) show that
$g_{1,j-1}(z),\ldots ,g_{n,j-1}(z)$ are in $S_{j-1}$.
Thus by Lemma \ref{lem: SN} we see that
\[ p^j Q_{i,j}(z)\ = \ g_{i,j-1}(z+1)-\varphi_i(g_{1,j-1}(z),\ldots
,g_{n,j-1}(z))\] is in $T_{j-1}$ (note that our hypothesis (b) implies
that each 
$$\varphi_i(g_{1,j-1}(z),\ldots ,g_{n,j-1}(z))$$
is a convergent power series).  It follows that we can write $p^j Q_{i,j}(z)=
c_{i,j}+ \sum_{k=1}^{\infty} p^k q_{ijk}(z)$ for some $c_{i,j}\in
\bZ_p$ and polynomials $q_{ijk}(z)\in B$ such that ${\rm
  deg}(q_{ijk})\le 2k-1$ for $k\le j-1$ and ${\rm deg}(q_{ijk})\le
2k-2$ for $k\ge j$.  Consequently, $p^jQ_{i,j}(z)$ is congruent modulo
$p^{j+1}C$ to the polynomial
\begin{equation}\label{one}
 c_{i,j}+ \sum_{k=1}^{j} p^k q_{ijk}(z),
\end{equation}
which is a polynomial of degree at most $2j-2$.  Note that the
polynomial in \eqref{one} must send $\bZ_p$ to $p^j \bZ_p$ since $p^j
Q_{i,j}(z)$ does, so dividing out by $p^j$, we see by Remark \ref{rmk}
that $Q_{i,j}(z)$ is congruent to a polynomial in $B$ of degree at
most $2j-2$ modulo $pC$.  To satisfy property (3) for $j$ it is
sufficient to find $\{h_{i, j}(z) \in B : 1 \le i \le n\}$ such that
\[
g_{i,j-1}(z+1)+p^j h_{i,j}(z+1)  - \varphi_i(g_{1,j-1}(z)+p^jh_{1,j}(z),\ldots ,
g_{n,j-1}(z)+p^jh_{n,j}(z))
\]
is in $p^{j+1}C$
for $1\le i\le n$.
Modulo $p^{j+1}C $, this expression becomes
\begin{equation*}
\begin{split}
p^jQ_{i,j}(z) + & p^j h_{i,j}(z+1) \\
& - p^j \sum_{\ell=1}^n h_{\ell,j}(z) \frac{\partial \varphi_{i}}
{\partial x_{\ell}} (x_1, ..., x_n)
\Big|_{x_{1}=g_{1,j-1}(z),\ldots, x_n=
g_{n,j-1}(z)}.
\end{split}
\end{equation*}
It therefore suffices to solve the system
\begin{equation} \label{eq: Q}
Q_{i,j}(z) + h_{i,j}(z+1)
 - \sum_{\ell=1}^n h_{\ell,j}(z) \frac{\partial \varphi_{i}}
{\partial x_{\ell}}(x_1, ..., x_n) \Big|_{x_{1}=g_{1,j-1}(z),\ldots, x_n=
g_{n,j-1}(z)}=0
\end{equation}
modulo $pC$ for $0 \leq i \leq n$.  
By our hypotheses (a) on $\varphi$, we have
 $$ \frac{\partial \varphi_{i}}
 {\partial x_{\ell}}(x_1, ..., x_n) \Big|_{x_{1}=g_{1,j-1}(z),\ldots,
   x_n= g_{n,j-1}(z)} \equiv \delta_{i\ell} ~(\bmod pC),$$ where
 $\delta_{i\ell}$ is usual Kronecker delta defined by
 $\delta_{i\ell}=0$ for $i\ne\ell$ and $\delta_{ii}=1$. Hence it is
 sufficient to solve
\begin{equation} \label{eq: Q'}
Q_{i,j}(z) + h_{i,j}(z+1)
 -  h_{i,j}(z)\equiv 0\pmod{pC}.
\end{equation}
Using the identity
\[ \binom{z+1}{k} - \binom{z}{k} \ = \binom{z}{k-1}, \]
we see that since $Q_{1,j},\ldots ,Q_{n,j}$ have degree at most $2j-2$ modulo $pC$,
there exists a  solution $[h_{1,j}(z), ..., h_{n, j}(z)] \in B^n$ satisfying conditions (1) and (2).
Thus conditions (1)-(3) are satisfied for $j$.  This completes the induction
step.

We now set
\[
f_i(z) :=  \sum_{j=0}^{\infty} p^j h_{i, j}(z).
\]
For $j\ge 1$, each $h_{i,j}(z)$ is of degree at most $2j-1$, and so
\[
h_{i,j}(z) = \sum_{k=0}^{2j-1}  c_{ijk} \binom{z}{k},
\]
with $c_{ijk} \in \bZ_p$, and $c_{ij0}=0$ for $j\ge 1$. Thus $f_i(0)=\omega_i$ for each $i=1,\dots,n$ because $h_{i,0}(z)=\omega_i$.
We now find that
\begin{eqnarray}
f_i(z) &=& \omega_i + \sum_{j=1}^{\infty} p^j 
\left( \sum_{k=1}^{2j-1} c_{ijk} \binom{z}{k}\right) \nonumber\\
&=& \omega_i+ \sum_{k=1}^{\infty} b_{ik}  \binom{z}{k} \label{eq475}
\end{eqnarray}
in which
\[
b_{ik} := \sum_{j=1}^{\infty} p^j c_{ijk}
\]
is absolutely convergent $p$-adically, since each $c_{ijk} \in \bZ_p$.
To show the series (\ref{eq475}) is analytic on $\ZZ_p$, we must establish that $|b_{ik}|_p/|k!|_p \to 0$
as $k \to \infty$, i.e. that for any $j >0$ one has $b_{ik}/k! \in p^j \bZ_p$
for all sufficiently large $k$ (see \cite[Theorem 4.7, pp. 354.]{Robert}).
To do this, we note that $c_{ijk}=0$ if $j<(k+1)/2$ since $\deg
h_{i,j} \leq 2j -1$.  Hence
\[ b_{ik} := \sum_{j\ge (k+1)/2} p^j c_{ijk}.\]
It follows that $|b_{ik}|_p\le p^{-(k+1)/2}$. Since $1/|k!|_p<p^{k/(p-1)}$, we see 
that 
$b_{ik}/k!\rightarrow 0$ since $p>3$.  Hence $f_1,\ldots ,f_n$ are $p$-adic analytic maps on $\bZ_p$.
 
By construction
\[ 
f_i(z) \equiv g_{i,j}(z)~(\bmod~p^{j+1}C).
\]
It then follows from property (3) above that
\[
f_i(z+1) \equiv \varphi_i( f_1(z), ..., f_n(z)) (\bmod~ p^{j+1} C)
\]
Since this holds for all $j \ge 1$, we conclude
that
\[
f_i(z+1) = \varphi_i( f_1(z), ..., f_n(z)),
\]
as desired.
\end{proof}

\section{Proof of the main result}
\label{main proof}

First we prove Theorem~\ref{real main result} if $X$ is an irreducible smooth variety.
\begin{theorem}
\label{main result}
Let $\Phi:X \lra X$ be an unramified endomorphism of an irreducible
smooth quasiprojective variety defined over $\bC$. Then for any
subvariety $V$ of $X$, and for any point $\alpha\in X(\bC)$ the
intersection $V(\C)\cap \cO_{\Phi}(\alpha)$ is a union of at most
finitely many orbits of the form $\cO_{\Phi^N}(\Phi^{\ell}(\alpha))$
for some nonnegative integers $N$ and $\ell$.
\end{theorem}

\begin{remark}
  Because $\Phi$ is unramified, while $X$ is smooth, we actually have
  that $\Phi$ is \'{e}tale (according to \cite[Theorem 5, page
  145]{Shaf}, we obtain an induced isomorphism on the tangent space at
  each point, which means that $\Phi$ is \'{e}tale).
\end{remark}

\begin{proof}[Proof of Theorem~\ref{main result}.]
  Let $V$ be a fixed subvariety of $X$.  We choose an embedding over
  $\bC$ of $\rho:X \lra \bP^M$ as an open subset of a projective
  variety (for some positive integer $M$).  We write $\rho(X) =
  Z(\fa)\setminus Z(\fb)$ for homogeneous ideals $\fa$ and $\fb$ in
  $\bC[x_0, \dots, x_M]$, where $Z(\fc)$ denotes the Zariski closed
  subset of $\bP^M$ on which the ideal $\fc$ vanishes.  We choose
  generators $F_1, \dots, F_m$ and $G_1, \dots, G_n$ for $\fa$ and
  $\fb$, respectively.  Let $R$ be a finitely generated $\bZ$-algebra
  containing the coefficients of the $F_i$, $G_i$, and of the
  polynomials defining the variety $V$, and such that
  $\alpha\in\bP^M(R)$. Let $\cX\subset\bP^M_{\Spec(R)}$ be the model
  for $X$ over $\Spec R$ defined by $Z(\fa')\setminus Z(\fb')$ where
  $\fa'$ and $\fb'$ are the homogeneous ideals in $R[x_0, \dots, x_N]$
  defined by $F_1, \dots, F_m$ and $G_1, \dots, G_n$, respectively;
  similarly, let $\cV$ be the model of $V$ over $\Spec(R)$.  We now
  prove two propositions which allow us to pass from $\bC$ to $\bZ_p$.

\begin{prop}
\label{extension to U}
  There exists a dense open subset $U$ of $\Spec(R)$ such that the
  following properties hold
\begin{enumerate}
\item there is a scheme $\cX_U$ that is smooth and quasiprojective over $U$, and whose generic fiber equals $X$;
\item each fiber of $\cX_U$ is geometrically irreducible;
\item $\Phi$ extends to an unramified endomorphism $\Phi_U$ of $\cX_U$; and
\item $\alpha$ extends to a section $U\lra \cX_U$.
\end{enumerate}
\end{prop}
\begin{proof}[Proof of Proposition~\ref{extension to U}.]
  Let $\mathcal{B}$ be the closed subset of $\cX$ which is the zero
  set of the polynomials defining $\Phi$. Because $\Phi$ is a
  well-defined morphism on the generic fiber $X$, we conclude that
  $\mathcal{B}$ does not intersect the generic fiber of $\cX\lra
  \Spec(R)$.  Therefore $\mathcal{B}$ is contained in the pullback
  under $\bP^M_{\Spec(R)}\lra\Spec(R)$ of a proper closed subset $E_1$
  of $\Spec(R)$. Similarly, let $\mathcal{C}$ be the closed subset
  defined by the intersection of $Z(\fb')$ with the Zariski closure of
  $\alpha$ in $\bP^M_{\Spec(R)}$. Because $\alpha\in X$, then
  $\mathcal{C}$ is contained in the pullback under
  $\bP^M_{\Spec(R)}\lra\Spec(R)$ of a proper closed subset $E_2$ of
  $\Spec(R)$. Let $U' = \Spec R \setminus (E_1 \cup E_2)$, let $\cX'$
  be the restriction of $\cX$ above $U'$, and let $\Phi_{U'}$ be the
  base extension of $\Phi$ to an endomorphism of $\cX'$.
  
  There is an open subset of $\cX'$ on which the restriction of the
  projection map to $\cX'\lra U'$ is smooth, by \cite[Remark VII.1.2,
  page 128]{AK} and there is an open subset of $\cX'$ on which
  $\Phi_{U'}$ is unramified by \cite[Proposition VI.4.6, page
  116]{AK}. Also, \cite[Theorem (2.10)]{vdD} shows that the condition
  of being geometrically irreducible is a first order property which
  is thus inherited by fibers above a dense open subset of $\Spec(R)$.
  Since each of these open sets contains the generic fiber, the
  complement of their intersection must be contained in the pullback
  under $\cX' \lra U'$ of a proper closed subset $E_3$ of $\Spec(R)$.
  Let $U:=U'\setminus E_3$; then letting $\cX_U$ be the restriction of
  $\cX'$ above $U$ yields the model and the endomorphism $\Phi_U$
  (which is the restriction of $\Phi_{U'}$ above $U$) with the desired
  properties.
\end{proof}

The following result is an easy consequence of \cite[Lemma 3.1]{Bell} (see also \cite{Lech}).  

\begin{prop}
\label{embedding again}
There exists a prime $p\ge 5$, an embedding of $R$ into $\bZ_p$, and a
$\bZ_p$-scheme $\cX_{\bZ_p}$ such that
\begin{enumerate}
\item $\cX_{\bZ_p}$ is smooth and quasiprojective over $\bZ_p$, and its
  generic fiber equals $X$;
\item both the generic and the special fiber of $\cX_{\bZ_p}$ are geometrically irreducible;
\item $\Phi$ extends to an unramified endomorphism $\Phi_{\bZ_p}$ of $\cX_{\bZ_p}$; and
\item $\alpha$ extends to a section $ \Spec \bZ_p \lra \cX_{\bZ_p}$.
\end{enumerate}
\end{prop}

\begin{proof}[Proof of Proposition~\ref{embedding again}.]
  With the notation as in Proposition~\ref{extension to U}, since $U$
  is a dense open subset of $\Spec(R)$, there exists a nonzero element
  $f\in R$ which is contained in all prime ideals of $\Spec R
  \setminus U$.  Therefore, there is an open affine subset of $U$ that
  is isomorphic to $\Spec(R[f^{-1}])$ (note that $R$ is an integral
  domain).  Since $R$ is finitely generated as a ring, then
  $R[f^{-1}]$ is as well and we may write $R[f^{-1}] := \bZ[u_1,
  \dots, u_e]$ for some nonzero elements $u_i$ of $\Frac(R)$.  By
  \cite[Lemma $3.1$]{Bell}, there is a prime $p$ such that $\Frac(R)$
  embeds into $\bQ_p$ in a way that sends all of the $u_i$ into
  $\bZ_p$, and $6$ is a unit in $\bZ_p$ (in particular, this last
  condition yields that $p\ge 5$). Thus, we obtain a map $\Spec \bZ_p
  \lra \Spec R[f^{-1}] \lra U$.  We let
  $\cX_{\bZ_p}:=\cX_{U}\times_{U}\Spec(\bZ_p)$ and let $\Phi_{\bZ_p}$
  be the base extension of $\Phi_U$ to an endomorphism of
  $\cX_{\bZ_p}$. The prime $(p)$ in $\bZ_p$ pulls back to a prime
  $\fq$ of $R[f^{-1}]$, so there is an isomorphism between
  $R[f^{-1}]/\fq$ and $\bZ_p / (p)$.  The fiber of $\cX_U$ is
  geometrically irreducible at $\fq$ by Proposition~\ref{extension to
    U}, so it follows that the fiber of $\cX_{\bZ_p}$ at $(p)$ is
  geometrically irreducible as well.  Since the properties of being
  quasiprojective, smooth, and unramified are all preserved by base
  extension (see Propositions VI.3.5 and VII.1.7 of \cite{AK}), our
  proof is complete (note that the embedding of $R[f^{-1}]$ into
  $\bZ_p$ automatically extends $\alpha$ to a section
  $\Spec(\bZ_p)\lra \cX_{\bZ_p}$).
\end{proof}  

For the sake of simplifying the notation, we let $\cX:=\cX_{\bZ_p}$
and let $\Phi$ denote the $\bZ_p$-endomorphism $\Phi_{\bZ_p}$ of
$\cX_{\bZ_p}$ constructed in Proposition~\ref{embedding again}. Also,
we use $\alpha$ to denote the section $\Spec(\bZ_p)\lra \cX(\bZ_p)$
induced in Proposition~\ref{embedding again};
i.e. $\alpha\in\cX(\bZ_p)$. Finally, we use $\cV=\cV_{\bZ_p}$ to
denote the $\bZ_p$-scheme which is the Zariski closure in $\cX$ of the
original subvariety $V$ of $X$.

Since the special fiber $\bcX$ of $\cX$ has finitely many $\bF_p$ points,
some iterate of $\alpha$ under $\Phi$ is in a periodic residue class
modulo $p$. At the expense of replacing $\alpha$ by a suitable
iterate under $\Phi$, we may assume that the residue class of
$\alpha$ is $\Phi$-periodic, say of period $N$ (note that replacing
$\alpha$ by one of its iterates under $\Phi$ will not change the
conclusion of Theorem~\ref{main result}). For each $j=0,\dots,N-1$, we
let $\vec{\alpha_j}:=\iota_j(\Phi^j(\alpha))\in\bZ_p^g$, where
$\iota_j^{-1}$ is the analytic bijection (defined as in
Section~\ref{geometry}) between $\bZ_p^g$ and the set of points of
$\cX(\bZ_p)$ which have the same reduction modulo $p$ as
$\Phi^j(\alpha)$.

Fix $j\in\{0,\dots,N-1\}$. Let $\cF_j:\bZ_p^g\lra\bZ_p^g$ be the
analytic function constructed in Proposition~\ref{power} corresponding
to $\Phi^N$ and to the residue class of $\Phi^j(\alpha)$; this
function satisfies
\begin{equation}
\label{logarithm}
\iota_j(\Phi^N(\Phi^j(\alpha)))=\cF_j(\vec{\alpha_j}).
\end{equation}
As proved in Proposition~\ref{pI}, there exists a positive integer $M_j$ such that
\begin{equation}
\label{congruence needed}
\cF_j^{M_j}(\vec{\beta})\equiv\vec{\beta}\pmod{p},
\end{equation}
for each $\vec{\beta}\in\bZ_p^g$.
Let $\ell\in\{0,\dots,M_j-1\}$ be fixed. Iterating \eqref{logarithm} yields 
\begin{equation}
\label{expanded}
\iota_j\left(\Phi^{N(M_jk+\ell)+j}(\alpha)\right)=\cF_j^{M_jk}(\cF_j^{\ell}(\vec{\alpha_j})) 
\end{equation}
for each $k\in\N$.  Using \eqref{congruence needed} (together with
Remark~\ref{divisible by many}) along with the fact that $p\ge 5$,
Theorem~\ref{thm: padic} implies that for each $j=0,\dots,N-1$ and for
each $\ell=0,\dots,M_j-1$, there exists an analytic function
$U_{j,\ell}:\bZ_p\lra\bZ_p^g$ such that
\begin{equation}
\label{ecu:1}
U_{j,\ell}(0)=\cF_j^{\ell}(\vec{\alpha_j})\in\bZ_p^g
\end{equation}
and
\begin{equation}
\label{ecu:2}
U_{j,\ell}(z+1)=\cF_j^{M_j}(U_{j,\ell}(z)),
\end{equation}
for each $z\in\bZ_p$. By \eqref{ecu:1} and \eqref{ecu:2}, we thus obtain
$$U_{j,\ell}(k)=\cF_j^{M_jk}(\cF_j^{\ell}(\vec{\alpha_j})),$$
for each $k\in\N$. Hence, \eqref{expanded} gives
\begin{equation}
\label{close to final}
U_{j,\ell}(k)=\iota_j\left(\Phi^{N(M_jk+\ell)+j}(\alpha)\right),
\end{equation}
for each $k\in\N$.  Then for each polynomial
$h\in\bZ_p[x_1,\dots,x_M]$ in the vanishing ideal of $\cV$, the
analytic function $h\circ \iota_j^{-1} \circ U_{j,\ell}$ defined on
$\bZ_p$ has infinitely many zeros $k\in\bZ_p$ if and only if it
vanishes identically (see \cite[Section $6.2.1$]{Robert}). Thus, for
each $j=0,\dots,N-1$ and each $\ell=0,\dots,M_j-1$, either
$$\cV(\bZ_p)\cap\cO_{\Phi^{NM_j}}(\Phi^{N\ell+j}(\alpha))\text{ is finite,}$$
or
$$\cO_{\Phi^{NM_j}}(\Phi^{N\ell+j}(\alpha))\subset \cV(\bZ_p).$$
This concludes the proof of Theorem~\ref{main result}.
\end{proof}

Now we extend the proof of Theorem~\ref{main result} to cover the case
of non-smooth varieties $X$, and thus prove our main result.
\begin{proof}[Proof of Theorem~\ref{real main result}.]
  First we show that it suffices to assume $X$ is irreducible. Indeed,
  since $\Phi$ is \'{e}tale, it permutes the irreducible components
  of $X$; hence, there exists a positive integer $N$ such that for
  each irreducible component $Y$ of $X$, the restriction of $\Phi^N$
  on $Y$ is an \'{e}tale endomorphism of $Y$. For each
  $\ell=0,\dots,(N-1)$, we let $Z_{\ell}$ be an irreducible component
  of $X$ containing $\Phi^{\ell}(\alpha)$. Then for each
  $\ell=0,\dots,(N-1)$ we have
$$V(\C)\cap \cO_{\Phi^N}(\Phi^{\ell}(\alpha))=(V\cap Z_{\ell})(\C)\cap \cO_{\Phi^N}(\Phi^{\ell}(\alpha)).$$
Since $\cO_{\Phi^N}(\Phi^{\ell}(\alpha))\subset Z_{\ell}$ and
$Z_{\ell}$ is irreducible, we are done.

We proceed now by induction on $\dim(X)$. If $\dim(X)=0$, there is nothing to prove.

Assume $\dim(X)=d\ge 1$, and that our result holds for all varieties
of dimension less than $d$. As shown above, we may assume $X$ is
irreducible. An \'{e}tale map sends smooth points into smooth points,
and non-smooth points into non-smooth points (because it induces an
isomorphism on the tangent space at each point).  If $\alpha$ (and
therefore $\cO_{\Phi}(\alpha)$) is in the smooth locus $X_{\smooth}$
of $X$, then Theorem~\ref{main result} finishes our proof because
$X_{\smooth}$ is also irreducible. If $\cO_{\Phi}(\alpha)$ is
contained in $X_1:=X\setminus X_{\smooth}$, then the inductive
hypothesis finishes our proof, because $\dim(X_1)<d$ and the
restriction of $\Phi$ to $X_1$ is \'{e}tale, since the base extension
of an \'etale morphism is \'etale (see \cite[VI.4.7.(iii)]{AK}).
\end{proof}

\section{Applications}
\label{applications}

We prove a Mordell-Lang type statement for automorphisms of any Noetherian integral scheme, which has several interesting consequences.

First, we define the \emph{full} orbit of a point $\alpha\in X(\C)$ under an automorphism $\Phi:X\lra X$ as the set 
$$\cOO_{\Phi}(\alpha):=\{\Phi^n(\alpha)\text{ : }n\in\Z\}.$$

\begin{theorem}
\label{automorphisms}
Let $X$ be any quasiprojective variety defined over $\C$, and let $\Phi:X\lra X$ be an automorphism. 
Then for each $\alpha\in X(\C)$, and for each subvariety $V\subset X$, the intersection $V(\C)\cap \cOO_{\Phi}(\alpha)$ is a union of at most finitely many points and at most finitely many full orbits of the form $\cOO_{\Phi^k}(\Phi^{\ell}(\alpha))$, for some positive integers $k$ and $\ell$.
\end{theorem}

\begin{proof}
Because $\Phi$ is \'{e}tale, we may apply Theorem~\ref{main result} and derive that the intersection $\cO_{\Phi}(\alpha)\cap V(\C)$ is a union of at most finitely many points and at most finitely many orbits of the form $\cO_{\Phi^k}(\Phi^{\ell}(\alpha))$ for some $k,\ell\in\N$ with $k\ge 1$. Same conclusion holds if we intersect $V$ with the orbit $\cO_{\Phi^{-1}}(\alpha)$. So, in order to prove Theorem~\ref{automorphisms} it suffices to show the following Claim.
\begin{claim}
\label{back and forth}
Let $Y$ be any quasiprojective variety defined over a field $K$ of arbitrary characteristic, let $\Psi:Y\lra Y$ be any automorphism, and let $\beta\in Y(K)$. Then the Zariski closure $W$ of $\cO_{\Psi}(\beta)$ contains $\cOO_{\Psi}(\beta)$.
\end{claim}

\begin{proof}[Proof of Claim~\ref{back and forth}.]
If $\beta$ is $\Psi$-periodic, then the conclusion is immediate because $\cO_{\Psi}(\beta)=\cOO_{\Psi}(\beta)$. 

Assume now that $\cO_{\Psi}(\beta)$ is infinite.
For each $n\in\N$, we let $W_n:=\Psi^n(Y)$. Then $W_n$ is the Zariski closure of the set of all $\Psi^m(\beta)$ for $m\ge n$, and so, $W_n$ contains all positive dimensional irreducible components of $W$. On the other hand, $W_n\subset W$ and  $\Psi^m(\beta)\in W\setminus W_n$ if and only if $m<n$ and $\Psi^m(\beta)$ is not contained in a positive dimensional irreducible component of $W$. Because $W$ has finitely many irreducible components, there exists $N\in\N$ such that for each $n\ge N$, the point $\Psi^n(\beta)$ is contained in a positive dimensional irreducible component of $W$. Therefore $W_{N+1}=W_N$, which means that $\Psi^{N+1}(W)=\Psi^N(W)$, and using the fact that $\Psi$ is an automorphism, we conclude that $\Psi(W)=W$. In particular, this shows that $\cOO_{\Psi}(\beta)\subset W$, as desired.
\end{proof}
We apply the conclusion of Claim~\ref{back and forth} to any automorphism $\Phi^k$ (for some nonzero integer $k$) and to any point $\Phi^{\ell}(\alpha)$ (for some integer $\ell$) such that $\cO_{\Phi^k}(\Phi^{\ell}(\alpha))\subset V$; this yields the conclusion of Theorem~\ref{automorphisms}.
\end{proof}

Special cases of Theorem~\ref{automorphisms} that have been previously
treated include the following:
\begin{enumerate}
\item $X$ is any commutative algebraic group, $\Phi$ is the
  translation-by-$P$-map on $X$, where $P\in X(\bC)$, and $\alpha=0\in
  X(\C)$ (this is a conjecture by Lang \cite{dio-geo}, which was first
  proved by Cutkosky and Srinivas in \cite[Theorem 7]{SriCut}). In
  particular, when $X$ is a semiabelian variety, we obtain the case of
  cyclic groups in the classical Mordell-Lang conjecture (proved by
  Faltings \cite{Faltings} and Vojta \cite{V1}).
\item $\Phi$ is any automorphism of an affine variety $X$ (this was
  proved by Bell in \cite{Bell}).
\item $\Phi$ is any automorphism of projective space (this was proved
  by Denis in \cite{Denis-dynamical}).
\end{enumerate}

Theorem~\ref{automorphisms} also answers to a question first posed in
\cite[Question 11.6]{KeeRogSta}. The following concept was introduced
in \cite[Definition 3.6]{KeeRogSta} (see also \cite[Section
5]{SriCut}).
\begin{definition}
\label{critically dense}
Let $S$ be an infinite set of closed points of an integral scheme $X$. Then we say that $S$ is critically dense if every infinite subset of $S$ has Zariski closure equal to $X$.
\end{definition}

\begin{question}[Keeler, Rogalski, Stafford]
\label{is critically dense?}
If $X$ is an irreducible quasiprojective variety over a field of characteristic $0$, and $\Phi:X\lra X$ is an automorphism, is every dense orbit $\{\Phi^i(\alpha)\text{ : }i\in\Z\}$ critically dense?
\end{question}

Theorem~\ref{automorphisms} provides a positive answer to
Question~\ref{is critically dense?} because it implies that any
subvariety $V\subset X$ containing infinitely many points of
$\cOO_{\Phi}(\alpha)$ must contain a set of the form
$\cOO_{\Phi^k}(\Phi^{\ell}(\alpha))$ for some $k\ge 1$. So,
$$\cOO_{\Phi}(\alpha)\subset \bigcup_{i=0}^{k-1} \Phi^i(V).$$
Because we assumed that $\cOO_{\Phi}(\alpha)$ is Zariski dense in $X$, we conclude that
$$X\subset\bigcup_{i=0}^{k-1} \Phi^i(V),$$
which can only happen if $V=X$ (because $X$ is
irreducible). Therefore, $\cOO_{\Phi}(\alpha)$ is indeed critically
dense.

A similar argument proves our Corollary~\ref{alternative formulation}.

As noted in the introduction Zhang \cite{ZhangLec} has conjectured
that if $f:X \lra X$ is what he calls a ``polarizable'' map defined
over a number field, then some point $X(K)$ has a dense orbit.  While
automorphisms are not polarizable in general, there are examples of
automorphisms of varieties $X$ for which {\it every} nonpreperiodic
point of $X$ has a dense orbit.  We describe a particular family of
examples below.

\begin{example}
  In \cite{SilK3}, Silverman studies a family of $K3$ surfaces $X$
  which have infinitely many automorphisms.  He considers a group of
  automorphisms $\cA$ generated by involutions $\sigma_1$, $\sigma_2$
  such that $\sigma_1^2 = \sigma_2^2 = \Id$ and $\sigma_1 \sigma_2$ has
  infinite order.  Silverman defines the {\it chain} containing a
  point $\alpha \in X(\bC)$ as the set of all $\tau \alpha$ where
  $\tau \in \cA$ and shows that any infinite chain is Zariski dense in
  $X$ (this is done by showing that there are no curves in $X$ that
  are preperiodic under the action of $\sigma_1 \sigma_2$).  Since the
  chain containing a point $\alpha$ is simply the union of the orbit
  of $\alpha$ under the action of $\sigma_1 \sigma_2$ with the orbit
  of $\sigma_1 \alpha$ under the action of $\sigma_1 \sigma_2$ it
  follows from Corollary~\ref{alternative formulation}, that any
  infinite subset of a chain is Zariski dense in $X$.  In particular,
  the orbit of any point $\alpha \in X(\bC)$ under the action of
  $\sigma_1 \sigma_2$ is critically dense.  It also follows from
  Theorem~\ref{automorphisms} that for any subvariety $V$ of $X$ and
  any point $\alpha \in X(\bC)$, the set of $\tau \in \cA$ such that
  $\tau(\alpha) \in V$ is a union of at most finitely many cosets of subgroups of
  $\cA$.
\end{example}

\def\cprime{$'$} \def\cprime{$'$} \def\cprime{$'$} \def\cprime{$'$}
\providecommand{\bysame}{\leavevmode\hbox to3em{\hrulefill}\thinspace}
\providecommand{\MR}{\relax\ifhmode\unskip\space\fi MR }
\providecommand{\MRhref}[2]{%
  \href{http://www.ams.org/mathscinet-getitem?mr=#1}{#2}
}
\providecommand{\href}[2]{#2}


\begin{thebibliography}{vdDS84}

\bibitem[AK70]{AK}
A.~Altman and S.~Kleiman, \emph{Introduction to {G}rothendieck duality theory},
  Lecture Notes in Mathematics, Vol. 146, Springer-Verlag, Berlin, 1970.

\bibitem[AM69]{AM}
M.~F. Atiyah and I.~G. Macdonald, \emph{Introduction to commutative algebra},
  Addison-Wesley Publishing Co., Reading, Mass.-London-Don Mills, Ont., 1969.

\bibitem[Bel06]{Bell}
J.~P. Bell, \emph{A generalised {S}kolem-{M}ahler-{L}ech theorem for affine
  varieties}, J. London Math. Soc. (2) \textbf{73} (2006), no.~2, 367--379.

\bibitem[BG06]{BG}
E.~Bombieri and W.~Gubler, \emph{Heights in {D}iophantine geometry}, New
  Mathematical Monographs, vol.~4, Cambridge University Press, Cambridge, 2006.

\bibitem[BGKT]{Par}
R.~L. Benedetto, D.~Ghioca, P.~Kurlberg, and T.~J. Tucker, \emph{The dynamical
  {M}ordell-{L}ang conjecture}, submitted for publication, 2007, available
  online at {\tt http://arxiv.org/abs/0712.2344}.

\bibitem[Bou06]{Bour2}
N.~Bourbaki, \emph{\'{E}l\'ements de math\'ematique. {A}lg\`ebre commutative.
  {C}hapitres 8 et 9}, Springer, Berlin, 2006, Reprint of the 1983 original.

\bibitem[CS93]{SriCut}
S.~D. Cutkosky and V.~Srinivas, \emph{On a problem of {Z}ariski on dimensions
  of linear systems}, Ann. of Math. (2) \textbf{137} (1993), no.~3, 531--559.

\bibitem[Den94]{Denis-dynamical}
L.~Denis, \emph{G\'{e}om\'{e}trie et suites r\'{e}currentes}, Bull. Soc. Math.
  France \textbf{122} (1994), no.~1, 13--27.

\bibitem[Fal94]{Faltings}
G.~Faltings, \emph{The general case of {S}. {L}ang's conjecture}, Barsotti
  Symposium in Algebraic Geometry (Abano Terme, 1991), Perspect. Math., no.~15,
  Academic Press, San Diego, CA, 1994, pp.~175--182.

\bibitem[GT]{newlog}
D.~Ghioca and T.~J. Tucker, \emph{Periodic points, linearizing maps, and the
  dynamical mordell-lang problem}, submitted for publication, 2008, available
  online at {\tt http://arxiv.org/abs/0805.1560}.

\bibitem[GTZ08]{Mike}
D.~Ghioca, T.~J. Tucker, and M.~Zieve, \emph{Intersections of polynomial
  orbits, and a dynamical {M}ordell-{L}ang conjecture}, Invent. Math.
  \textbf{171} (2008), no.~2, 463--483.

\bibitem[Har77]{H}
R.~Hartshorne, \emph{Algebraic geometry}, Springer-Verlag, New York, 1977.

\bibitem[HY83]{HY}
M.~Herman and J.-C. Yoccoz, \emph{Generalizations of some theorems of small
  divisors to non-archimedean fields}, Geometric dynamics (Rio de Janeiro,
  1981), Lecture Notes in Math., no. 1007, Springer-Verlag, Berlin, 1983,
  pp.~408--447.

\bibitem[KRS05]{KeeRogSta}
D.~S. Keeler, D.~Rogalski, and J.~T. Stafford, \emph{Na\"{i}ve noncommutative
  blowing up}, Duke Math. J. \textbf{126} (2005), no.~3, 491--546.

\bibitem[Lan83]{dio-geo}
S.~Lang, \emph{Fundamentals of diophantine geometry}, Springer-Verlag, New
  York, 1983.

\bibitem[Lec53]{Lech}
C.~Lech, \emph{A note on recurring series}, Ark. Mat. \textbf{2} (1953),
  417--421.

\bibitem[Mah35]{Mahler-2}
K.~Mahler, \emph{Eine arithmetische {E}igenshaft der {T}aylor-{K}oeffizienten
  rationaler {F}unktionen}, Proc. Kon. Nederlandsche Akad. v. Wetenschappen
  \textbf{38} (1935), 50--60.

\bibitem[Mah58]{MahlerK1}
\bysame, \emph{An interpolation series for continuous functions of a {$p$}-adic
  variable}, J. Reine Angew. Math. \textbf{199} (1958), 23--34.

\bibitem[Mah61]{MahlerK2}
\bysame, \emph{A correction to the paper ``an interpolation series for
  continuous functions of a {$p$}-adic variable''}, J. Reine Angew. Math.
  \textbf{208} (1961), 70--72.

\bibitem[Mat86]{Mats2}
H.~Matsumura, \emph{Commutative ring theory}, Cambridge Studies in Advanced
  Mathematics, vol.~8, Cambridge University Press, Cambridge, 1986, Translated
  from the Japanese by M. Reid.

\bibitem[Ray83a]{Raynaud1}
M.~Raynaud, \emph{Courbes sur une vari\'{e}t\'{e} ab\'{e}lienne et points de
  torsion}, Invent. Math. \textbf{71} (1983), no.~1, 207--233.

\bibitem[Ray83b]{Raynaud2}
\bysame, \emph{Sous-vari\'{e}t\'{e}s d'une vari\'{e}t\'{e} ab\'{e}lienne et
  points de torsion}, Arithmetic and geometry, vol. I, Progr. Math., vol.~35,
  Birkh\"{a}user, Boston, MA, 1983, pp.~327--352.

\bibitem[RL03]{Riv1}
J.~Rivera-Letelier, \emph{Dynamique des fonctions rationnelles sur des corps
  locaux}, Ast\'erisque (2003), no.~287, 147--230, Geometric methods in
  dynamics. II.

\bibitem[Rob00]{Robert}
A.~M. Robert, \emph{A course in {$p$}-adic analysis}, Graduate Texts in
  Mathematics, vol. 198, Springer-Verlag, New York, 2000.

\bibitem[Sha77]{Shaf}
I.~R. Shafarevich, \emph{Basic algebraic geometry}, study ed., Springer-Verlag,
  Berlin, 1977, Translated from the Russian by K. A. Hirsch, Revised printing
  of Grundlehren der mathematischen Wissenschaften, Vol. 213, 1974.

\bibitem[Sil91]{SilK3}
J.~H. Silverman, \emph{Rational points on {$K3$} surfaces: a new canonical
  height}, Invent. Math. \textbf{105} (1991), no.~2, 347--373.

\bibitem[Sko34]{Skolem}
T.~Skolem, \emph{Ein {V}erfahren zur {B}ehandlung gewisser exponentialer
  {G}leichungen und diophantischer {G}leichungen}, C. r. 8 congr. scand. \`{a}
  Stockholm (1934), 163--188.

\bibitem[Ull98]{Ullmo}
E.~Ullmo, \emph{Positivit\'e et discr\'etion des points alg\'ebriques des
  courbes}, Ann. of Math. (2) \textbf{147} (1998), no.~1, 167--179.

\bibitem[vdDS84]{vdD}
L.~van~den Dries and K.~Schmidt, \emph{Bounds in the theory of polynomial rings
  over fields. {A} nonstandard approach}, Invent. Math. \textbf{76} (1984),
  no.~1, 77--91.

\bibitem[Voj96]{V1}
P.~Vojta, \emph{Integral points on subvarieties of semiabelian varieties. {I}},
  Invent. Math. \textbf{126} (1996), no.~1, 133--181.

\bibitem[Zha98]{Zhang}
S.~Zhang, \emph{Equidistribution of small points on abelian varieties}, Ann. of
  Math. (2) \textbf{147} (1998), no.~1, 159--165.

\bibitem[Zha06]{ZhangLec}
S.~Zhang, \emph{Distributions in {A}lgebraic {D}ynamics}, Survey in
  Differential Geometry, vol.~10, International Press, 2006, pp.~381--430.

\end{thebibliography}

\end{document}